\documentclass[11pt,reqno]{amsart}
\usepackage[top=1in, bottom=1in, left=1in, right=1in]{geometry}
\usepackage{times, amsthm, amssymb, amsmath, amsfonts, amsthm, bm, graphicx, mathrsfs}
\usepackage{latexsym,wasysym}
\usepackage[usenames,dvipsnames]{color}
\usepackage{tikz}
\usetikzlibrary{fadings}
\usepackage{graphicx}
\usepackage{relsize}
\usepackage{xypic}
\usepackage{float}
\usepackage{setspace}
\usepackage{array}
\usepackage{multirow}
\usepackage{multicol}
\usepackage{mathtools}
\usepackage{enumerate}
\usepackage{bbm}

\newtheorem{theorem}{Theorem}[section]
\newtheorem{lemma}[theorem]{Lemma}
\newtheorem{proposition}[theorem]{Proposition}
\newtheorem{corollary}[theorem]{Corollary}

\theoremstyle{definition}
\newtheorem{definition}[theorem]{Definition}

\theoremstyle{remark}

\newtheorem*{notation}{Notation}

\newcommand{\bbN}{\mathbbm{N}}

\newcommand{\bbP}{\mathbbm{P}}

\newcommand{\bbid}{\mathbbm{1}}
\newcommand{\bbk}{\mathbbm{k}}

\newcommand{\calA}{{\mathcal{A}}}

\newcommand{\calF}{{\mathcal{F}}}

\newcommand{\calS}{{\mathcal{S}}}

\newcommand{\frakm}{{\mathfrak{m}}}
\newcommand{\frakq}{{\mathfrak{q}}}

\newcommand{\rank}{\operatorname{rank}}
\newcommand{\depth}{\operatorname{depth}}
\newcommand{\Tor}{\operatorname{Tor}}
\newcommand{\Hilb}{\operatorname{Hilb}}
\newcommand{\init}{\operatorname{in}}

\newcommand{\del}{{\partial}}

\tikzfading[name=fade right, 
    left color=transparent!0, 
    right color=transparent!100]
\tikzfading[name=fade left, 
    left color=transparent!100, 
    right color=transparent!0]
    
\def\block(#1,#2)#3{\multicolumn{#2}{c}{\multirow{#1}{*}{$ #3 $}}}

\begin{document} 

\title{Toward Free Resolutions over Scrolls}

\author{Laura Felicia Matusevich}
\address{Mathematics Department\\Texas A\&M University\\College Station, TX 77843}
\email{laura@math.tamu.edu}
\author{Aleksandra Sobieska}
\email{ola@math.tamu.edu}

\thanks{The authors were partially supported by NSF grant DMS--1500832.}

\date{}

\begin{abstract}
Let $R=\bbk[x]/I$ where $I$ is the defining ideal of a rational normal
$k$-scroll. We compute the Betti numbers of the ground
field $\bbk$ as a module over $R$. For $k=2$, we give the minimal
free resolution of $\bbk$ over $R$. 
\end{abstract}

\parindent 0pt
\parskip .7em

\maketitle 

\section{Introduction} 

Free resolutions are a mainstay in commutative algebra, as they
contain a wealth of information about the object resolved. A free
resolution is an extended presentation: if a module is given by
generators and relations, the resolution records also relations among
the relations, relations among the relations of the relations, and so
on. In the special case of modules over the polynomial ring over a
field, Hilbert's Syzygy Theorem guarantees that this process 
always terminates. Furthermore, there are algorithms to compute
resolutions over polynomial rings, that are implemented in computer
algebra systems. In some special cases where additional structure is
present, such as for monomial ideals in polynomial rings, 
free resolutions can be given combinatorially. Free
resolutions over polynomial rings have been the focus of intense study;
over more general rings however, free resolutions are typically infinite, and
are consequently harder to work with.

If $R$ is a standard graded $\bbk$-algebra, where $\bbk$ is a field,
it is important to understand 
the resolution of $\bbk$ as an $R$-module. One reason is that, for any
$R$-module $M$, the rank of the $i$th free module in a \emph{minimal} free
resolution of $M$, called its $i$th \emph{Betti number}, equals
$\dim_\bbk\Tor^R_{i}(\bbk,M)$, which can be 
computed from a free resolution of $\bbk$.
Such a ring $R$ is \emph{Koszul} if $\bbk$ has a \emph{linear}
free resolution over $R$, that is, if the entries of the differentials
in a resolution of $\bbk$ as an $R$-module are linear forms. The
Koszul property has received much attention in combinatorial
settings. An early result~\cite{fquadratic} states that if $R =
\bbk[x_1,\dots,x_n]/I$, where $I$ is generated by monomials of degree
two, then $R$ is Koszul. By a degeneration argument, if $R =
\bbk[x_1,\dots,x_n]/J$ where $J$ has a quadratic initial ideal, then
$R$ is Koszul. For semigroup rings, a characterization of the Koszul
property is an open problem, see~\cite{pifrotr} for a survey of known
results on resolutions over semigroup rings. 

In many cases of rings that are known to be Koszul, the
resolution of the residue field is not explicitly known. For
semigroup rings, we are aware only of resolutions over the
rings associated to rational normal curves~\cite{ghphfotr}. In
fact,~\cite{ghphfotr} gives the minimal free resolution for any
monomial ideal in this case.

In this article, we consider the next class of examples after rational
normal curves, namely rational normal scrolls. We compute the Betti
numbers of the residue field (Theorem~\ref{bettimain}), and for
$2$-scrolls, we give its minimal free resolution. 

We illustrate our results in an example. Consider $R = \bbk[x_1,\dots,x_6]/I$,
where $I$ is the ideal of $2\times 2$ minors of the matrix
\[
\left[
\begin{array}{cc|cc}
x_1 & x_2 & x_4 & x_5 \\ 
x_2 & x_3 & x_5 & x_6
\end{array}
\right].
\]
The ideal $I$ gives the defining equations of the rational normal scroll
$\calS(2,2)$. In this case, the minimal free resolution of $\bbk$ over
$R$ is
\[
\cdots \to R^{64\cdot 3^{i-3}} \xrightarrow{\del_{i}} \cdots
\xrightarrow{\del_4} R^{64} \xrightarrow{\del_3}
R^{21}\xrightarrow{\del_2} R^6 \xrightarrow{ [x_1 \; x_2 \; \cdots \; x_6]} R
\to \bbk \to 0
\]
The matrices giving the differentials $\del_i$ are highly
structured. Throughout this article, we adopt the following notations: 
$\mathbf{0}^{p \times q}$ denotes a zero matrix of size $p\times q$;
where it causes no confusion, zero blocks or entries of a matrix are
indicated by $\mathbf{0}$ or simply left empty; $\bbid_\ell$ is the
$\ell\times \ell$ identity matrix; direct sum of matrices denotes 
concatenation of blocks along the main diagonal (with off-diagonal
blocks equal to zero). 
With these conventions, 
\[
\del_2=
\scriptsize{\left[ 
\begin{array}{c|c|c|c|rrrrr}
\cline{1-1}
\multirow{2}{*}{\: \: \normalsize{$\varphi_0$} \: \:} & \multicolumn{3}{c|}{} &x_4&&&& \\
\cline{2-2}
& \multirow{2}{*}{\: \: \normalsize{$\varphi_0$} \: \:} & \multicolumn{2}{c|}{} &&x_4&&& \\ 
\cline{1-1}
& & \multicolumn{2}{c|}{} &&&x_4&x_5&x_6 \\
\hline
\multicolumn{2}{c|}{} & \multirow{2}{*}{\: \: \normalsize{$\varphi_0$} \: \:} & &-x_1&-x_2&-x_3&& \\
\cline{4-4} 
\multicolumn{2}{c|}{} & & \multirow{2}{*}{\: \: \normalsize{$\varphi_0$} \: \:} & &&&-x_3& \\ 
\cline{3-3}
\multicolumn{3}{c|}{} & & &&&&-x_3 \\
\cline{4-4}
\end{array}
\right]},
\]
where 
$\varphi_0 = \scriptsize{\left[ 
		\begin{array}{rr|rr}
		{x}_{2}&{x}_{3}&{x}_{5}&{x}_{6}\\
		{-{x}_{1}}&{-{x}_{2}}&{-{x}_{4}}&{-{x}_{5}}

		\end{array} \right] }$;

\[
\del_3 = \left[ 
 \begin{array}{c|c|c|c|c|}
 \multirow{2}{*}{ \larger[1]{$\varphi_1^{\oplus 4}$} } & \multicolumn{2}{c|}{ \larger[2]{x_4 \cdot \bbid_8} }& \multicolumn{2}{c}{} \\
 \cline{2-5} 
 \multicolumn{1}{c|}{} & \multicolumn{2}{c|}{} & \multicolumn{2}{c}{ \larger[2]{-x_3 \cdot \bbid_8} } \\
 \hline 
 \multicolumn{1}{c|}{} & \multirow{2}{*}{\: \: \large{$-\varphi_0$} \: \:} & \multicolumn{3}{c}{} \\
 \cline{3-3}
 \multicolumn{1}{c|}{} & \multicolumn{1}{c|}{} & \multirow{2}{*}{\: \: \large{$-\varphi_0$} \: \:} & \multicolumn{2}{c}{} \\
 \cline{2-2} \cline{4-4}
 \multicolumn{1}{c|}{} & \multicolumn{1}{c|}{}& & \multirow{2}{*}{\: \: \large{$-\varphi_0$} \: \:} & \multicolumn{1}{c}{}\\
 \cline{3-3} \cline{5-5}
 \multicolumn{1}{c|}{} & \multicolumn{2}{c|}{} & & \multirow{2}{*}{\: \: \large{$-\varphi_0$} \: \:} \\
 \cline{4-4}
 \multicolumn{1}{c|}{} & \multicolumn{3}{c|}{} & \\
 \cline{5-5}
 \end{array}
\right]
\]
where
$\varphi_1 = \scriptsize{\left[ 
		\begin{array}{rrrr|rrrr|rrrr}
		{x}_{2}&{x}_{3}&{x}_{5}&{x}_{6}&{x}_{4}&0&0&0&0&0&0&0\\
		{-{x}_{1}}&{-{x}_{2}}&{-{x}_{4}}&{-{x}_{5}}&0&{x}_{4}&{x}_{5}&{x}_{6}&0&0&0&0\\
		\hline 
		0&0&0&0&{-{x}_{1}}&{-{x}_{2}}&{-{x}_{3}}&0&{x}_{2}&{x}_{3}&{x}_{5}&{x}_{6}\\
		0&0&0&0&0&0&0&{-{x}_{3}}&{-{x}_{1}}&{-{x}_{2}}&{-{x}_{4}}&{-{x}_{5}}\\
		\end{array} \right] }$; and for $i \geq 4$,
\[
\del_i = \left[
\begin{array}{c|c}
\multirow{1}{*}{ \large $\varphi_{i-2}^{\oplus 4}$ } & 
	\begin{array}{c|c}
	x_4 \cdot \bbid_{8 \cdot 3^{i-3}} & \\ 
	\hline
	& -x_3 \cdot \bbid_{8 \cdot 3^{i-3}}
	\end{array} \\
\hline
 & \multirow{1}{*}{\large $- \varphi_{i-3}^{\oplus 4}$}
\end{array}
\right]
\]
where
\[
\varphi_2 = \left[
\begin{array}{c|cccccccccccc|c}
\multicolumn{1}{c|}{\multirow{4}{*}{\larger{$\varphi_1$}}} & 
& & & &
\multicolumn{8}{c}{\multirow{4}{*}{}} & 
\multicolumn{1}{|c}{\multirow{4}{*}{}} \\ 
& 
&  &  &  & 
& & & & & & & 
& \\ 
& 
x_1 & x_2 & x_4 & x_5 & 
& & & & & & & 
& \\ 
& 
0 & 0 & 0 & 0 & 
& & & & & & & 
& \\ 
\hline 
\multicolumn{1}{c|}{\multirow{4}{*}{}} & 
-x_2 & -x_3 & -x_5 & -x_6 & 
& & & & & & & &
\multicolumn{1}{c}{\multirow{4}{*}{}} \\ 
& 
x_1 & x_2 & x_4 & x_5 & 
-x_2 & -x_3 & -x_5 & -x_6 & 
& & & 
& \\ 
& 
& & & & 
x_1 & x_2 & x_4 & x_5 & 
-x_2 & -x_3 & -x_5 & -x_6 
& \\ 
& 
& & & &
& & & & 
x_1 & x_2 & x_4 & x_5 
& \\
\hline 
\multicolumn{1}{c|}{\multirow{4}{*}{}} & 
\multicolumn{8}{c}{\multirow{4}{*}{}} & 
0 & 0 & 0 & 0 &
\multicolumn{1}{c}{\multirow{4}{*}{\larger{$\varphi_1$}}} \\ 
& 
& & & & & & & & 
-x_2 & -x_3 & -x_5 & -x_6
& \\ 
& 
& & & & & & & & 
& & & 
& \\ 
& 
& & & & & & & & 
& & & 
& \\         
\end{array}
\right] \in R^{12 \times 36}. 
\]
and $\varphi_i = \varphi_{i-1} \bigoplus \varphi_{i-2}^{\oplus 3}
\bigoplus \varphi_{i-1}$ for $i \geq 3$.
\subsection*{Outline}
This article is organized as follows. Section~\ref{sec:Preliminaries}
contains necessary background. In Section~\ref{sec:Betti} we compute
the Betti numbers of rational normal
$k$-scrolls. Section~\ref{sec:Resolution} is devoted to constructing
the minimal resolution of $\bbk$ over a $2$-scroll.

\subsection*{Acknowledgements}
We thank Christine Berkesch and Chris O'Neill for inspiring
conversations while we worked on this project.

\section{Preliminaries}
\label{sec:Preliminaries}

We work in $n = \sum\limits_{i=1}^k m_i$ variables, and denote the
polynomial ring by $S = \bbk[x_{i,j} \, | \, 1 \leq i \leq k, 1 \leq j
\leq m_i ]$. The \emph{rational normal $k$-scroll} $\calS(m_1-1,
m_2-1, \ldots, m_k-1)$ is the 
variety in $\bbP^{n-1}$ defined by the ideal $I_2(M)$ of $2 \times 2$
minors of the $2 \times (n-k)$ matrix  
\begin{equation}
\label{eqn:minorMatrix}
M = \left[ 
\begin{array}{ccc|ccc|c|ccc}
x_{1,1} & \ldots & x_{1,m_1-1} & x_{2,1} & \ldots & x_{2,m_2-1} & \ldots \ldots & x_{k,1} & \ldots & x_{k,m_k-1} \\ 
x_{1,2} & \ldots & x_{1, m_1} & x_{2,2} & \ldots & x_{2,m_2} & \ldots \ldots & x_{k,2} & \ldots & x_{k,m_k} 
\end{array}
\right] .
\end{equation}

Throughout this article, we often forego writing ``rational normal" and call $\calS(m_1-1,
m_2-1, \ldots, m_k-1)$ a $k$-scroll and $\calS(m-1, n-m-1)$ a scroll. 

When $k=1$, $\calS(n-1)$ is a \emph{rational normal curve}, that is, the
variety defined by $2\times 2$ the minors of the matrix  
\[ 
\left[
\begin{array}{cccc}
x_1 & x_2 & \ldots & x_{n-1} \\ 
x_2 & x_3 & \ldots & x_n
\end{array}
\right] .
\] 
%

\subsection{Koszul algebras}

Let $A = \bigoplus_{i \geq 0} A_i$ be a standard graded
$\bbk$-algebra, and let $\beta_i^A(\bbk)$ be the $i$th Betti number of
$\bbk$ as an $A$-module. We consider the \emph{Poincar\'e series}
$P_A(t)$ of $A$, and its \emph{Hilbert series} $\Hilb(A;t)$, defined
as follows
\[
P_A(t) =
\sum\limits_{i\geq 0} \beta_i^A(\bbk) \cdot t^i
\quad \text{and} \quad
\Hilb(A;t) = \sum\limits_{i \geq 0}
\dim_\bbk A_i \cdot t^i.
\]
When $A$ is a \emph{Koszul} ring, there is a strong
relationship between these two series, that is useful later on. The
following result can be taken as a definition.

\begin{theorem}{(cf. \cite[Definition-Theorem~1]{fkoszulalgs})}\label{koszulalgs}
	A graded algebra $A$ is Koszul if and only if the following equivalent conditions are satisfied:
	\begin{enumerate}
		\item the minimal graded $A$-resolution of $\bbk$ is linear.
		\item $\Hilb(A;-t)P_A(t)=1$.
	\end{enumerate} 
\end{theorem}

As we mentioned in the introduction, the rings that are studied in this article are Koszul.

\begin{theorem}
For $M$ as in~\eqref{eqn:minorMatrix}, $R = S/I_2(M)$ is a Koszul ring.
\end{theorem}

\begin{proof}
By~\cite[Theorem~2.2]{bhvsyzandwalks}, a sufficient condition for a quotient $\bbk[x_1,
\ldots, x_n]/I$ to be Koszul is the existence of a homogeneous
quadratic Gr{\"o}bner basis for $I$. It follows that $R$ is Koszul, since
the $2\times 2$ minors of $M$ form a Gr{\"o}bner basis for $I_2(M)$ with respect to a
reverse lexicographic ordering (see~\cite[Lemma~2.2]{kpudivsonrns}).
\end{proof}

\subsection{Semigroup Rings}

Let $\calA = \{\gamma_1, \ldots, \gamma_n \} \subseteq \bbN^d$. We also use
$\calA$ to denote the $d \times n$ matrix with columns
$\gamma_1,\dots,\gamma_n$. We assume that $d \leq n$ and $\rank \calA = d$. The
configuration (or matrix) $\calA$ induces a map
\begin{align*}
\bbk[x_1, \ldots, x_n] & \rightarrow \bbk[t_1, \ldots, t_d] \\ 
x_i & \mapsto \mathbf{t}^{\gamma_i} = t_1^{\gamma_{i,1}} \cdots t_d^{\gamma_{i,d}} 
\end{align*}
The kernel $I_\calA = \langle \mathbf{x}^u - \mathbf{x}^v \, | \, \calA u =
\calA v \rangle$ of this map is a prime binomial ideal called the
\emph{toric ideal} associated to $\calA$. 
The \emph{semigroup ring} associated to $\calA$ is 
\[
\bbk[\mathbf{t}^{\gamma_1}, \mathbf{t}^{\gamma_2}, \ldots, \mathbf{t}^{\gamma_n}] \cong \bbk[x_1, \ldots, x_n] / I_{\calA}.
\]

By~\cite[Lemma~2.1]{pgbovomd}, $I_2(M) =I_\calA$, where $\calA$ is
the $(k+1) \times n$ matrix  
\begin{equation}
\label{eqn:A} 
\calA = \left[ \begin{array}{cccc|cccc|c|cccc}
1 & \cdots & \cdots & 1 & 0 & \cdots & \cdots & 0 & 0 \: \cdots \cdots\: 0 & 0 & \cdots & \cdots & 0 \\ 
0 & \cdots & \cdots & 0 & 1 & \cdots & \cdots & 1 & 0 \: \cdots \cdots\: 0 & 0 & \cdots & \cdots & 0 \\ 
\vdots & & & \vdots & & & & & & \vdots & & & \vdots \\ 
0 & \cdots & \cdots & 0 & 0 & \cdots & \cdots & 0 & 0 \: \cdots \cdots\: 0 & 1 & \cdots & \cdots & 1 \\ 
0 & 1 & \cdots & m_1-1 & 0 & 1 & \cdots & m_2-1 & \cdots \cdots & 0 & 1 & \cdots & m_k-1
\end{array} \right],
\end{equation}
so that $R=S/I_2(M)$ is a semigroup ring.
%

\section{Betti Numbers of $\bbk$ over $k$-scrolls}
\label{sec:Betti}

Our first main theorem gives the Betti numbers of the field $\bbk$
over $R=S/I_2(M)$, where $M$ is as in~\ref{eqn:minorMatrix}. 

\begin{theorem}\label{bettimain}
	Let $I_2(M)$ define the rational normal $k$-scroll
        $\calS(m_1-1, \ldots, m_k-1)$. If $R = S/I_2(M)$, then the
        $i$th Betti number of $\bbk$ as an $R$-module is
	\[ 
	\beta_i^R(\bbk) = \sum\limits_{j=0}^{i} \binom{k+1}{j}(n-k-1)^{i-j}
	\] 
	In particular $\beta_{k+r}^R(\bbk) = (n-k-1)^{r-1}(n-k)^{k+1}$ for $r \geq 0$. 
\end{theorem}

Because $R$ is Koszul, Theorem~\ref{koszulalgs} implies that we can
obtain the Poincar{\'e} series of $R$ by inverting its Hilbert
series. Since Hilbert series are preserved under Gr\"obner
degeneration, it is enough to compute the Hilbert series of
$S/\init_\prec(I_2(M))$ for $\prec$ a monomial order in $S$. This
task is easiest if we are fortunate enough that our ideal has a squarefree
initial ideal. The next result states that this is indeed the case for scrolls.

\begin{theorem}
\label{thm:initialIdeal}
Let $\prec$ be the lexicographic monomial order on $S$ given by $x_{1,1}
\succ x_{1,2} \succ \ldots \succ x_{1,m_1} \succ x_{2,1} \succ \ldots
\succ x_{k,m_k}$, then
\begin{equation}
\label{eqn:initialIdeal}
\init_\prec(I_2(M)) = \langle x_{i,j}x_{i,\ell} \mid |j-\ell| \geq 2 \rangle +  \langle x_{i,j}x_{r,s} \mid 1 \leq i < r \leq k, 1 \leq j < m_i, 1 < s \leq m_r \rangle,
\end{equation}
that is, $\init_\prec(I_2(M))$ is generated by the products of
variables on the main diagonals of $M$. In particular,
$\init_\prec(I_2(M))$ is a squarefree monomial ideal.
\end{theorem}

Denote by $D$ the ideal on the right hand side of~\eqref{eqn:initialIdeal}.
To prove Theorem~\ref{thm:initialIdeal}, we begin by pinpointing
which monomials are \emph{not} in $D$. 

\begin{lemma}\label{lemxunotinD} 
	Suppose $x^u\notin D$. 
	\begin{enumerate}[a)]
		\item If there exists $i$ such that $x^u$ contains two variables with first index $i$ with nonzero exponents, then $u$ is of the form 
		\[
		u = (0 \ldots 0 \, a_1 | 0 \ldots 0 \, a_2 | \ldots | 0 \ldots 0 \, a_{i-1} | 0 \ldots 0 \, a_{i,\ell} \, a_{i,\ell+1} \, 0 \ldots 0 | a_{i+1} \, 0 \ldots 0 | \ldots | a_k \, 0 \ldots 0 ) 
		\]
		\item Otherwise, $u$ is of the form 
		\[
		u = (0 \ldots 0 \, a_1 | 0 \ldots 0 \, a_2 | \ldots | 0 \ldots 0 \, a_{i-1} | a_i \, 0 \ldots 0 | a_{i+1} \, 0 \ldots 0 | \ldots | a_k \; 0 \ldots 0)
		\] 
	\end{enumerate}
\end{lemma} 

\begin{proof}
	The lemma follows from these observations. 
	\begin{enumerate}[i)]
		\item If $x^u$ contains the variables $x_{i,j}, x_{i,\ell}$ with $j < \ell$ both with nonzero exponent, then $\ell = j+1$. Consequently, $x^u$ cannot contain 3 variables from the same block with nonzero exponent. 
		\item If $x^u$ contains variables $x_{i,j}, x_{r,s}$ with $i < r$ and $j < m_i$, both with nonzero exponent, then $s=1$. 
	\end{enumerate}
\end{proof}

The following result is used to show that $D$ is equal to $\init_\prec I_2(M)$. 

\begin{proposition}\label{propxusuccxv}
        Let $\calA$ be as in~\eqref{eqn:A} (so that $I_2(M)=I_\calA$).
	If $x^u \notin D$, $x^u \succ x^v$, and $\calA u = \calA v$, then $u=v$. 
\end{proposition}

\begin{proof}
	In Lemma~\ref{lemxunotinD}, case b) is a special case of a) where
        $\ell = 1$ and $a_{i,\ell+1} = 0$, so we may assume $u$
        satisfies case a). We also assume $u \neq 0$, and write $v =
        (b_{1,1}, b_{1,2}, \ldots, b_{k,m_k})$.  
	
	Suppose $a_1 \neq 0$. Since $x^u \succ x^v$, the monomial $x^v$ cannot
        contain any variable greater than $x_{1,m_1}$. Then, as $\calA u
        = \calA v$, $x^u$ and $x^v$ must contain the same power of
        $x_{1,m_1}$. The same argument implies that $x^u$ and $x^v$
        contain the same powers of all variables up to and including
        $x_{i-1,m_{i-1}}$.  
	
	Now again, since $x^u \succ x^v$ lexicographically,
        $a_{i,\ell} \geq b_{i,\ell}$ and $b_{i,0} = \ldots =
        b_{i,\ell-1} =0$. As $\calA u = \calA v$, we have $a_{i,\ell}
        + a_{i,\ell+1} = b_{i,\ell} + b_{i,\ell+1} + \ldots +
        b_{i,m_i}$. But if $a_{i,\ell} > b_{i,\ell}$, then $(\calA
        u)_{k+1} > (\calA v)_{k+1}$. This implies that $a_{i,\ell} =
        b_{i,\ell}$, and similarly $a_{i,\ell+1} = b_{i,\ell+1}$, so
        $b_{i,t} = 0$ for $t \geq \ell+2$.  
	
	To finish the proof, note that $(\calA u)_{k+1} = (m_1-1)a_1 +
        \ldots + (m_{i-1}-1)a_{i-1} + (\ell-1) a_{i,\ell} + \ell
        a_{i,\ell+1} = (m_1-1)b_{1,m_1} + \ldots +
        (m_{i-1}-1)b_{i-1,m_{i-1}} + (\ell-1) b_{i,\ell} + \ell
        b_{i,\ell+1}$. Because $\calA u = \calA v$, this implies that
        $b_{j,t} = 0$ for $j > i$ and $t > 1$. Again, using $\calA u =
        \calA v$, we conclude that $b_{j,1} = a_j$ for all $j > i$.  
\end{proof}

We are ready to prove Theorem~\ref{thm:initialIdeal}.

\begin{proof}[Proof of Theorem~\ref{thm:initialIdeal}]
	Since $I_2(M)$ is $\calA$-homogeneous, its initial
        ideal is generated by the initial forms of $\calA$-homogeneous
        elements of $I_2(M)$.
If $P \in I_2(M)$ is $\calA$-homogeneous and $\init_\prec P \notin D$, then
$P$ has one term by Proposition~\ref{propxusuccxv}. But since $I_2(M)$
is a toric ideal, it contains no monomials, so that such a $P$ cannot belong to $I_2(M)$. We
conclude that if $P$ is  $\calA$-homogeneous and $P \in I_2(M)$, then
$\init_\prec P \in D$.
\end{proof}

With a squarefree initial ideal in hand, we now turn to
Stanley--Reisner theory. Let $\Delta$ be the simplicial
complex on the vertex set  $\{(i,j) \mid 1\leq i \leq k, 1\leq j \leq m_i\}$
whose Stanley--Reisner ideal is $D=\init_\prec I_2(M)$. 
By definition,
this means that $D$ is generated by monomials whose index sets
correspond to nonfaces of $\Delta$. It follows from Lemma~\ref{lemxunotinD}
that $\Delta$ is the simplicial complex whose maximal faces are
\begin{equation}
\label{eqn:facets}
\{ (1,m_1), (2,m_2), \ldots, (i,m_{i-1}), (i,j), (i,j+1), (i+1,1),
\ldots, (k,1) \}  \text{ for }  1 \leq i \leq k, \, 1 \leq j \leq m_i-1,
\end{equation}
in particular, $\Delta$ is pure of dimension $k$.
Figure~\ref{fig:simplicialcpx} illustrates this simplicial complex in
an example.

\begin{figure}[!h]
	\centering
	\begin{tikzpicture}[scale=1.5]
	\tikzstyle{vertex}=[circle,thick,draw=black,fill=black,inner sep=0pt,minimum width=4pt,minimum height=4pt]
	
	\fill[gray, opacity=.3] (0,0) -- (3,0) -- (2,1) -- (0,1);
	
	\node[vertex, label=below left: ${\begin{smallmatrix}(1,1)\end{smallmatrix}}$] (x1) at (0,0) {};
	\node[vertex, label=below: ${\begin{smallmatrix}(1,2)\end{smallmatrix}}$] (x2) at (1,0) {};
	\node[vertex, label=below: ${\begin{smallmatrix}(1,3)\end{smallmatrix}}$] (x3) at (2,0) {};
	\node[vertex, label=below right: ${\begin{smallmatrix}(1,4)\end{smallmatrix}}$] (x4) at (3,0) {};
	\node[vertex, label=above left: ${\begin{smallmatrix}(2,1)\end{smallmatrix}}$] (x5) at (0,1) {};
	\node[vertex, label=above: ${\begin{smallmatrix}(2,2)\end{smallmatrix}}$] (x6) at (1,1) {};
	\node[vertex, label=above right: ${\begin{smallmatrix}(2,3)\end{smallmatrix}}$] (x7) at (2,1) {};   
	
	\draw (x1) to (x2)
	(x2) to (x3)
	(x3) to (x4)
	(x4) to (x5)
	(x5) to (x6)
	(x6) to (x7)
	(x1) to (x5)
	(x2) to (x5)
	(x3) to (x5)
	(x4) to (x6)
	(x4) to (x7);             
	
	\end{tikzpicture}
	\caption{The simplicial complex $\Delta$ for $\calS(3,2)$}
	\label{fig:simplicialcpx}
\end{figure}
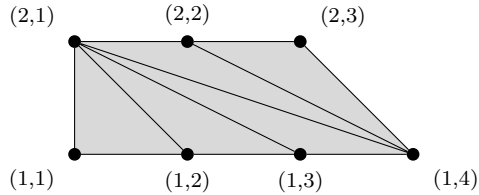

It is well known that the Hilbert series of a Stanley--Reisner
ring can be given in terms of the face numbers of the corresponding
simplicial complex. Explicitly,
\[
\Hilb(S/D;t) = \frac{1}{(1-t)^{k+1}} \sum\limits_{d=0}^{k+1} f_{d-1} t^d (1-t)^{k+1-d},
\]
where $f_d$ is the number of $d$-dimensional faces of $\Delta$. We now
compute these face numbers.

\begin{proposition}
\label{prop:faceNumbers}
	If $\Delta$ is the simplicial complex whose Stanley--Reisner
        ideal is $D$, then $f_d = \binom{k}{d}n - d \binom{k+1}{d+1}$
        for $d \geq -1$. In particular, the face numbers of $\Delta$
        depend only on $k$ and $n$, and not on $m_1,\dots,m_k$.
\end{proposition}

\begin{proof}
	We prove this by induction on $k$. Note that, by construction,
        $f_0=n$, regardless of the value of $k$. 

        If $k=1$, $\Delta$ has
        $f_1=n-1$ one-dimensional faces, namely $\{ (1,i), (1,i+1) \}$
        for $i =1,\dots,n-1$ (cf~\cite[Theorem~3.9]{prshtsam}).

	For the inductive step, let $\Delta$ be the complex associated
        to $\calS(m_1-1, \ldots, m_k-1)$ and $\Delta'$ be the complex
        associated to $\calS(m_1-1, \ldots, m_{k+1}-1)$. The complex
        $\Delta$ is naturally a subcomplex of $\Delta'$. We
        assume that
        $f_d(\Delta) = \binom{k}{d}(m_1 + \ldots + m_k) - d
        \binom{k+1}{d+1}$. Using the description of the facets of
        $\Delta$ from~\eqref{eqn:facets} (and the corresponding
        description for the facets of $\Delta'$) we see that the $d$-dimensional faces of $\Delta'$
        are: 
	\begin{itemize}
		\item $f_d(\Delta)$ $d$-dimensional faces of $\Delta$,
		\item $f_{d-1}(\Delta)$ faces of the form $\tau
                  \cup \{(k+1,1)\}$, where $\tau$ is a
                  $(d-1)$-dimensional face of $\Delta$,
		\item $\binom{k}{d}(m_{k+1}-1)$ faces with $d$ vertices
                  from the set $\{ (i,m_i) \mid 1 \leq i \leq k\}$ and one
                  vertex from $\{ (k+1,j) \mid 2 \leq j \leq
                  m_{k+1}\}$, and
		\item $\binom{k}{d-1}(m_{k+1}-1)$ faces with $d-1$
                  vertices from $\{(i,m_i) \mid 1 \leq i \leq k \}$
                  union an element of $\{ \{(k+1,j),(k+1,j+1) \} \mid
                  1 \leq j \leq m_{k+1}-1 \}$.
	\end{itemize}
	
	Adding these together and applying the inductive hypothesis
        yields $f_{d}(\Delta') = \binom{k+1}{d}(m_1 + \ldots +
        m_{k+1}) - d \binom{k+2}{d+2}$, as we wanted.
\end{proof}

The following result gives the Hilbert series of $R$; the proof is a straightforward, if hefty, bullying of binomial coefficients.

\begin{proposition}\label{prophilbertseries}
$\Hilb(R; -t) = \Hilb(S/D; -t) = \frac{1-(n-k-1)t}{(1+t)^{k+1}}$ $\qed$
\end{proposition}

We are finally ready to prove Theorem~\ref{bettimain}.

\begin{proof}[Proof of Theorem~\ref{bettimain}]
Since $R$ is a Koszul ring, it follows from
Proposition~\ref{prophilbertseries} that the Poincar{\'e} series of $R$ is 
\[
P_R(t) = \frac{1}{\Hilb(R; -t)} = \dfrac{(1+t)^{k+1}}{1-(n-k-1)t} =
\sum\limits_{i=0}^{\infty} \left[ \sum\limits_{j=0}^{i}
  \binom{k+1}{j}(n-k-1)^{i-j} \right] t^{i}.
\] 
For the last equality, we use 
$(1+t)^{k+1} = \sum\limits_{i=0}^{k+1} \binom{k+1}{i}t^i$ and
$\dfrac{1}{1-(n-k-1)t} = \sum\limits_{i=0}^{\infty} (n-k-1)^i t^i$.  
We conclude that $\beta_i^R(\bbk) = \sum\limits_{j=0}^{i}
  \binom{k+1}{j}(n-k-1)^{i-j}$.
The special formula for
$\beta_{k+r}(\bbk)$ follows from the simplification of this sum when $\binom{k+1}{j}$ becomes $0$. 
\end{proof}

\section{The Resolution of $\bbk$ for $k=2$}
\label{sec:Resolution}

One of the difficulties when dealing with infinite free resolutions
and unbounded Betti numbers is to give an explicit presentation for the
differentials. In the case $k=2$, the combinatorics of the ring $R$
ensure a strong block structure that makes giving explicit matrices
achievable.  

\begin{notation}
In the case $k=2$, we write $\calS(m-1,n-m-1)$ instead of
$\calS(m_1-1,m_2-1)$, and forego double indexing  to replace $x_{1,j}$
by $x_j$ and $x_{2,j}$ by $x_{m+j}$. Finally, we denote $p=n-m$.

With this new notation, the matrix~\eqref{eqn:minorMatrix} is replaced
by the  $2 \times (n-2)$ matrix
\[
M = \left[ 
\begin{array}{cccc|cccc}
x_1 & x_2 & \ldots & x_{m-1} & x_{m+1} & x_{m+2} & \ldots & x_{n-1} \\ 
x_2 & x_3 & \ldots & x_m & x_{m+2} & x_{m+3} & \ldots & x_{n}
\end{array}
\right],
\]
and the ideal $I_2(M)$ is the toric ideal $I_\calA$ associated to the $3 \times n$ matrix 
\[
\calA = \left[ \begin{array}{ccccc|cccc}
1 & 1 & 1 & \ldots & 1 & 0 & 0 & \ldots & 0 \\ 
0 & 0 & 0 & \ldots & 0 & 1 & 1 & \ldots & 1 \\ 
0 & 1 & 2 & \ldots & m-1 & 0 & 1 & \ldots & p-1
\end{array} \right].
\]
\end{notation}

Our ultimate goal is to construct the minimal free resolution of
$\bbk$ as an $R$-module.
Our point of departure is the short exact sequence
\begin{equation}
\label{eqn:shortExact}
0 
\rightarrow 
\langle x_1, \ldots, x_m \rangle \cap \langle x_{m+1}, \ldots, x_n \rangle 
\rightarrow 
\langle x_1, \ldots, x_m \rangle \oplus \langle x_{m+1}, \ldots, x_n \rangle
\rightarrow 
\langle x_1, \ldots, x_n \rangle 
\rightarrow 
0.
\end{equation}

We construct free resolutions $(\calF_\bullet(I_1),\del_{I_1,i})$,
$(\calF_\bullet(I_2),\del_{I_2,i})$, and $(\calF_\bullet,\del_{J,i})$
of the ideals $I_1 = \langle x_1, \ldots, x_m \rangle$, $I_2 = \langle
x_{m+1}, \ldots, x_n \rangle$ and $J = I_1 \cap I_2$ respectively. We
then combine these resolutions via mapping cone to make a resolution
of $\frakm = \langle x_1, \ldots, x_n \rangle$. 
Augmenting the resolution of $\frakm$ to be a resolution of $\bbk =
R/\frakm$ results in a shift of one step, and minimality is assured by
the previous Betti number computations. We obtain the resolution
\[ 
\calF_\bullet: \cdots \xrightarrow{\del_6} R^{(n-2)^3(n-3)^2} \xrightarrow{\del_5} R^{(n-2)^3(n-3)} \xrightarrow{\del_4} R^{(n-2)^3} \xrightarrow{\del_3} R^{n^2-3n+3} \xrightarrow{\del_2} R^n \xrightarrow{\del_1} R \xrightarrow{\del_0} \bbk \rightarrow 0.
\]
%

\subsection{The Differentials of $\calF_\bullet$}
Our first objective is to explicitly describe the differentials
$\del_i$ of $\calF_\bullet$. These differentials are induced by a
mapping cone. More precisely,
\[
\del_1 = 
[x_1 \; x_2 \; \cdots \; x_n], \;
\del_2 = \left[ \begin{array}{c|c}
\begin{array}{r|r} \del_{I_1,1} & \\ \hline & \del_{I_2,1} \end{array} & \alpha_0
\end{array} \right], \;
\del_{i+1} = \left[ \begin{array}{c|c}
\begin{array}{r|r} \del_{I_1,i} & \\ \hline & \del_{I_2,i} \end{array} & \alpha_{i-1} \\ 
\hline
\mathbf{0} & -\del_{J,i-1}
\end{array} \right] \textrm{ for all } i \geq 2.
\]

The maps $\alpha$ are the chain maps from $\calF_\bullet(J)$ to $\calF_\bullet(I_1) \oplus \calF_\bullet(I_2)$, which are:
\[
\begin{aligned} 
\alpha_0 &= \left[ \begin{array}{cc}
\begin{matrix} 
x_{m+1} & 0 & \ldots & 0 & \multicolumn{3}{c}{\multirow{4}{*}{$\mathbf{0}$ }} \\ 
0 & x_{m+1} & \ldots & 0 \\ 
0 & \ddots & \ddots & 0 \\ 
0 & 0 & \ldots & x_{m+1} &
x_{m+2} & x_{m+3} & \ldots & x_n 
\end{matrix}\\  
\hline
\begin{matrix} 
-x_1 & -x_2 & \ldots & -x_{m} & 0 & 0 & \cdots & 0 \\ 
\multicolumn{4}{c}{\multirow{3}{*}{$\mathbf{0}$ }} & -x_m & 0 & \cdots & 0 \\ 
& & & & 0 & -x_m & \cdots & 0 \\ 
& & & & 0 & \ddots & \ddots & 0 \\ 
& & & & 0 & 0 & \cdots & -x_m
\end{matrix}
\end{array} \right] \in R^{n \times (n-1)} \\
\alpha_i &= \left[ \begin{array}{cc}
x_{m+1} \cdot \bbid_{(m-1)(n-2)(n-3)^{i-1}} & 0 \\ 
0 & -x_m \cdot \bbid_{(p-1)(n-2)(n-3)^{i-1}} 
\end{array}
\right] \in R^{(n-2)^2(n-3)^{i-1} \times (n-2)^2(n-3)^{i-1}} 
\end{aligned} 
\] 

The constituent resolutions $\calF_\bullet(I_1)$,
$\calF_\bullet(I_2)$, and $\calF_\bullet(J)$ have highly structured
differentials, the building blocks of which are now given:  
\[
\begin{aligned}
\varphi_0 &= \begin{bmatrix*}[c]
x_2 & x_3 & \ldots & x_m & x_{m+2} & \ldots & x_n \\ 
-x_1 & -x_2 & \ldots & -x_{m-1} & -x_{m+1} & \ldots & -x_{n-1}
\end{bmatrix*} \in R^{2 \times (n-2)} \\
\Phi_d &= \begin{bmatrix*}[c] 
\varphi_0 & \mathbf{0}^{1 \times (n-2)} & \cdots & \cdots & \cdots & \mathbf{0}^{1 \times (n-2)}\\ 
\mathbf{0}^{1 \times (n-2)} & \varphi_0 & \mathbf{0}^{1 \times (n-2)} & \cdots & \cdots & \mathbf{0}^{1 \times (n-2)} \\
\vdots & \mathbf{0}^{1 \times (n-2)} & \varphi_0 & \mathbf{0}^{1 \times (n-2)} & \cdots & \mathbf{0}^{1 \times (n-2)} \\ 
\vdots & \vdots & \ddots & \ddots & \ddots & \vdots \\ 
\mathbf{0}^{1 \times (n-2)} & \cdots & \cdots & \cdots & \mathbf{0}^{1 \times (n-2)} & \varphi_0
\end{bmatrix*} \in R^{d \times (d-1)(n-2)}
\end{aligned}
\]
Note that $\Phi_d$ is very sparse and consists of block components,
but is not a block diagonal matrix. The structure of $\Phi_d$ is
illustrated in Figure~\ref{fig:Phimatrix}, with gray squares denoting non-zero entries, and
empty squares denoting 0. These nonzero $\varphi_0$-blocks appear
$d-1$ times. 

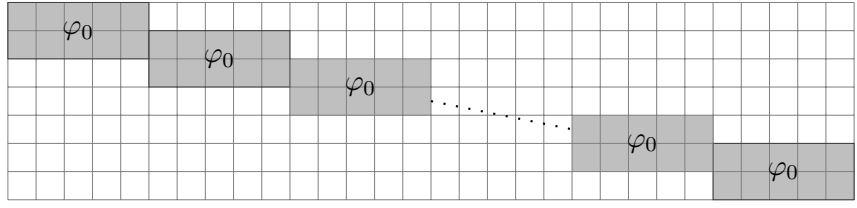
\begin{figure}[!h]
	\centering
	\begin{tikzpicture}[scale=.75]
	\draw[step=.5cm,gray,very thin] (-7.5,-1) grid (7.5,2.5);
	\foreach \x/\y in {-7.5/2.5, -5/2, 5/0} {
		\draw [fill=gray, opacity=.5] (\x,\y) rectangle (\x+2.5,\y-1);
	}
	\fill[gray, path fading=fade right, fading transform={rotate=-30}, opacity = .5] (-2.5,1.5) rectangle (0,.5);
	\fill[gray, path fading=fade left, fading transform={rotate=-30}, opacity=.5] (2.5,.5) rectangle (5,-.5);
	\draw[loosely dotted, thick] (0,.75) -- (2.5,.25);
	\node (a) at (-6.25,2) {$\varphi_0$};
	\node (b) at (-3.75,1.5) {$\varphi_0$};
	\node (c) at (-1.25,1) {$\varphi_0$};
	\node (d) at (3.75,0) {$\varphi_0$};
	\node (e) at (6.25,-.5) {$\varphi_0$};
	\end{tikzpicture}
	\caption{The structure of $\Phi_d$}
	\label{fig:Phimatrix}
\end{figure}
We denote by $u_i \in R^{(m-2)(n-2) \times (n-2)}$ and $v_i \in
R^{(p-2)(n-2) \times (n-2)}$ the following matrices, which are
almost entirely composed of zeros save for a single row that equals
the first row or second row of $M$, respectively. More precisely,  
\[
u_i = \begin{bmatrix}
\mathbf{0}^{i(n-2)+m-1 \times n-2} \\ 
\begin{matrix}
x_1 & \ldots & x_{m-1} & x_{m+1} & \ldots & x_{n-1}
\end{matrix}\\
\mathbf{0}^{((m-i-2)(n-2)-m) \times n-2}
\end{bmatrix}  \qquad
v_i = \begin{bmatrix} 
\mathbf{0}^{i(n-2)+m-2 \times n-2} \\ 
\begin{matrix}
-x_2 & \ldots & -x_m & -x_{m+2} & \ldots & -x_{n}
\end{matrix} \\ 
\mathbf{0}^{(((p-i-2)(n-2))-m+1) \times n-2}
\end{bmatrix} .
\]

Despite the length of the exponents, these matrices are simple: $u_i$
is the $(m-2)(n-2) \times (n-2)$ matrix with the top row of $M$ in the
$(i(n-2)+m)$-th row, and $v_i$ is the $(p-2)(n-2) \times (n-2)$ matrix
with the negative of the bottom row of $M$ in the $(i(n-2)+m-1)$-st
row.  

Finally, we introduce the following notation:
\[ 
\begin{aligned}
\varphi_1 &= \left[ 
\begin{array}{c|c|c}
\Phi_{m-1} & 
\begin{matrix} 
x_{m+1} & 0 & \ldots & 0 & \multicolumn{3}{c}{\multirow{4}{*}{$\mathbf{0}$ }} \\ 
0 & x_{m+1} & \ldots & 0 \\ 
0 & \ddots & \ddots & 0 \\ 
0 & 0 & \ldots & x_{m+1} &
x_{m+2} & x_{m+3} & \ldots & x_n 
\end{matrix} & 
\mathbf{0} \\ 
\hline
\mathbf{0} & 
\begin{matrix} 
-x_1 & -x_2 & \ldots & -x_{m-1} & -x_m & 0 & \cdots & 0 \\ 
\multicolumn{4}{c}{\multirow{3}{*}{$\mathbf{0}$ }} & 0 & -x_m & \cdots & 0 \\ 
& & & & 0 & \ddots & \ddots & 0 \\ 
& & & & 0 & 0 & \cdots & -x_m
\end{matrix} & 
\Phi_{p-1}
\end{array} 
\right] \in R^{(n-2) \times (n-2)(n-3)} \\
\varphi_2 &= \left[ 
\begin{array}{c|ccc|c}
	\bigoplus\limits_{m-2} \varphi_1 & 
	\begin{matrix}
		u_0 & 
		u_1 &
		\cdots &
		u_{m-3}
	\end{matrix} 
	& \mathbf{0} 
	& \mathbf{0} & 
	\mathbf{0} \\
	\hline
	\mathbf{0} & 
	& 
	-\Phi_{n-2} & 
	& 
	\mathbf{0} \\ 
	\hline
	\mathbf{0} & 
	\mathbf{0} & 
	\mathbf{0} &
	\begin{matrix}
		v_0 & 
		v_1 & 
		\cdots & 
		v_{p-3}
	\end{matrix} &  
	\bigoplus\limits_{p-2} \varphi_1
\end{array}
\right] \in R^{(n-2)(n-3) \times (n-2)(n-3)^2} \\ 
\varphi_i &= \varphi_{i-1}^{\oplus (m-2)} \bigoplus \varphi_{i-2}^{\oplus(n-3)} \bigoplus \varphi_{i-1}^{\oplus (p-2)} \in R^{(n-2)(n-3)^{i-1} \times (n-2)(n-3)^{i}} \textrm{    for } i \geq 3
\end{aligned}
\]

The presentation for $\varphi_2$ is perhaps deceiving; the brunt of
the matrix is a direct sum of $\varphi_1$'s. It is only (most of) the
middle $(n-3)(n-2)$ columns that have additional entries above or
below the middle $-\Phi_{n-2}$. 

Using these $\varphi$'s, we construct resolutions of $J$, $I_1$ and
$I_2$. The ideal $J$ has resolution over $R$ as shown below:  

\[
\calF_{\bullet}(J) : \cdots \xrightarrow{\del_{J,4}} R^{(n-2)^2(n-3)^2} \xrightarrow{\del_{J,3}} R^{(n-2)^2(n-3)} \xrightarrow{\del_{J,2}} R^{(n-2)^2} \xrightarrow{\del_{J,1}} R^{n-1} \xrightarrow{\del_{J,0}} J \rightarrow 0
\]
where 
\[ 
\begin{aligned}
\del_{J,0} &= \begin{bmatrix*}
x_1x_{m+1} & x_2x_{m+1} & \ldots& x_mx_{m+1} & x_mx_{m+2} & x_mx_{m+3} & \ldots & x_mx_n
\end{bmatrix*} \in R^{1 \times (n-1)} \\
\del_{J,1} &= \Phi_{n-1} \in R^{(n-1) \times (n-2)^2} \\
\del_{J,i} &= \varphi_{i-1}^{\oplus (n-2)} \in R^{(n-2)^2(n-3)^{i-2} \times (n-2)^2(n-3)^{i-1}} \textrm{    for } i \geq 2
\end{aligned}
\]

The ideal $I_1$ has resolution over $R$ as shown below: 
\[
\calF_{\bullet}(I_1) : \cdots \xrightarrow{\del_{I_1,4}} R^{(m-1)(n-2)(n-3)^2} \xrightarrow{\del_{I_1,3}} R^{(m-1)(n-2)(n-3)} \xrightarrow{\del_{I_1,2}} R^{(m-1)(n-2)} \xrightarrow{\del_{I_1,1}} R^{m} \xrightarrow{\del_{I_1,0}} I_1 \rightarrow 0
\]
where 
\[ 
\begin{aligned}
\del_{I_1,0} &= \begin{bmatrix*}
x_1 & x_2 & \ldots & x_m
\end{bmatrix*} \in R^{1 \times m}\, ; \qquad
\del_{I_1,1} = \Phi_{m} \in R^{m \times (m-1)(n-2)} \\
\del_{I_1,i} &= \varphi_{i-1}^{\oplus (m-1)} \in R^{(m-1)(n-2)(n-3)^{i-2} \times (m-1)(n-2)(n-3)^{i-1}} \textrm{    for } i \geq 2
\end{aligned}
\]

The ideal $I_2$ has resolution over $R$ as shown below: 
\[
\calF_{\bullet}(I_2) : \cdots \xrightarrow{\del_{I_2,4}} R^{(p-1)(n-2)(n-3)^2} \xrightarrow{\del_{I_2,3}} R^{(p-1)(n-2)(n-3)} \xrightarrow{\del_{I_2,2}} R^{(p-1)(n-2)} \xrightarrow{\del_{I_2,1}} R^{p} \xrightarrow{\del_{I_2,0}} I_2 \rightarrow 0
\]
where 
\[
\begin{aligned}
\del_{I_2,0} &= \begin{bmatrix*}
x_{m+1} & x_{m+2} & \ldots & x_n 
\end{bmatrix*} \in R^{1 \times p} \, ; \qquad
\del_{I_2,1} = \Phi_{p} \in R^{(p) \times (p-1)(n-2)} \\
\del_{I_2,i} &= \varphi_{i-1}^{\oplus(p-1)} \in R^{(p-1)(n-2)(n-3)^{i-2} \times (p-1)(n-2)(n-3)^{i-1}} \textrm{    for } i \geq 2
\end{aligned}
\]

Our main result is as follows.

\begin{theorem}
\label{thm:resolution}
$\calF_\bullet$ constructed above is the minimal free resolution of $\bbk$ over $R$.
\end{theorem}

\subsection{Outline of the proof of Theorem~\ref{thm:resolution}} 
The remainder of this section is devoted to showing that
$\calF_\bullet$ is indeed the minimal free resolution of $\bbk$ over
$R$. We now lay out the steps in this proof.

Most of the work goes to showing that $\calF_\bullet(I_1)$,
$\calF_\bullet(I_2)$ and $\calF_\bullet(J)$ are free resolutions of
$I_1$, $I_2$ and $J$ respectively. The matrices considered in these
three cases have very similar structure, and the details in proving
exactness are virtually identical. Thus, we give only the proof that
$\calF_\bullet(J)$ is a resolution. Exactness of $\calF_\bullet(J)$ is shown in Subsection~\ref{subsec:exact}, using ideas from~\cite{bewmace}. 

What remains is to provide the map of complexes  $\alpha: \calF_\bullet(J) \to
\calF_\bullet(I_1) \oplus \calF_\bullet(I_2)$ lifting the inclusion $J
\to I_1\oplus I_2$ from the short exact
sequence~\eqref{eqn:shortExact}. This is done in Subsection~\ref{subsec:mappingCone}.

Once $\alpha$ is constructed, the mapping cone procedure ensures that
$\calF_\bullet$ is exact, and thus a free resolution of $\bbk$. That
it is the \emph{minimal} free resolution of $\bbk$ follows by
inspection, or by Theorem~\ref{bettimain}.

\subsection{$\calF_\bullet(J)$ is exact} 
\label{subsec:exact}

We need generators for $J = \langle x_1, \ldots, x_m \rangle \cap \langle x_{m+1}, \ldots, x_n \rangle$. Clearly,
\[ J = 
\left \langle \begin{array}{cccc}
x_1x_{m+1}, & x_1x_{m+2}, & \ldots, & x_1x_n, \\
x_2x_{m+1}, & x_2x_{m+2}, & \ldots, & x_2x_n, \\ 
\vdots & \vdots & \vdots & \vdots \\
x_mx_{m+1} & x_mx_{m+2}, & \ldots, & x_mx_n
\end{array} \right \rangle .
\]

However, many of these monomials are equal in $R$; in fact, $x_ix_j =
x_kx_\ell$ if $i+j = k + \ell$, as long as $j \neq m$ and $k \neq
m+1$. This means that, in the above arrangement, all monomials on the same
skew-diagonal are the same, for example, $x_3x_{m+1} = x_2x_{m+2} =
x_1x_{m+3}$. 
Consequently,
\[ J = \langle x_1x_{m+1}, x_2x_{m+1}, \ldots, x_mx_{m+1}, x_mx_{m+2}, x_mx_{m+3}, \ldots, x_mx_n \rangle. \]
We start by checking that we are working with complexes.
\begin{proposition}\label{FJcomplex}
	$\calF_\bullet(J)$, $\calF_{\bullet}(I_1)$ and $\calF_\bullet(I_2)$ are complexes. 
\end{proposition}
\begin{proof}
This is a straightforward, if tedious, calculation.
A key observation is that $\varphi_i \circ \varphi_{i+1} = 0$ for all
$i$. This follows, as each of these compositions has entries that
are either $0$ or binomials in $I_2(M)$. Given the direct sum structure
of the differentials, this is enough to show our proposed differentials
compose to zero. 
\end{proof}

Our next goal is to show that $\calF_\bullet(J)$ is exact. (The same
argument, with minor modifications, shows the same for
$\calF_\bullet(I_1)$ and $\calF_\bullet(I_2)$.) We need some notation.

\begin{definition}
\label{def:rank}
Let $A$ be a noetherian commutative ring.
Let $f: F \rightarrow G$ be a map of free $A$-modules, which is
represented by a matrix with entries in $A$. The \emph{rank} of $f$ is
the size of the largest nonvanishing minor of this matrix. If $f$ has
rank $r$, we use $I(f)$ to denote the ideal generated by the $r\times
r$ minors of (the matrix representing) $f$.
\end{definition}

The following results are used to prove exactness.

\begin{lemma}{\cite[Lemma~20.10]{ecawav}}\label{exactlem}
	Let $A$ be a commutative noetherian ring.
        A complex $F~\xrightarrow{f}~G~\xrightarrow{g}~H$ of free
        $A$-modules with $I(f) = I(g) = A$ is exact iff $\rank f +
        \rank g = \rank G$.   
\end{lemma}

\begin{lemma}[Sylvester's Rank Inequality]\label{sylrankineq}
	If $U$ and $V$ are matrices with entries in a field, where $U$
        is $r \times s$ and $V$ is $s \times t$, then
	\[ 
	\rank U + \rank V - s \leq \rank UV.
	\] 
\end{lemma}

By Lemma~\ref{exactlem}, it is important to know the ranks of the
differentials of  $\calF_\bullet(J)$. Due to the block structure, 
we must first address the matrices $\varphi_i$. 

\begin{proposition}\label{phirank}
	The rank of $\Phi_d$ is $d-1$ for all $d \geq 2$ and the rank $\varphi_i$ is $(n-3)^i$ for all $i \geq 0$.
\end{proposition} 
\begin{proof}
	Because $R$ is a domain, 
        dependences among rows of a matrix over $R$ can be read off
        from the vanishing of minors. In fact, the rank of a
        matrix over
        $R$ equals the rank of that matrix over the field of fractions
        of $R$. In this proof, we work over the field of fractions of
        $R$, which gives us access to Lemma~\ref{sylrankineq}.
	
	It is clear that $\rank \varphi_0 = 1$, as all $2 \times 2$
        minors of $\varphi_0$ are exactly the same as the minors of
        $M$, which belong to $I_2(M)$.  
	
	Next we must show that $\rank \varphi_1 = n-3$. We know that
        the rank of $\varphi_1$ is at least $n-3$, as the
        minor of size $n-3$ corresponding to rows $2, 3, \ldots, n-2$ and columns $1, 1 +
        (n-2), 1 + 2(n-2), \ldots, 1 + (n-4)(n-2)$ equals
        $(-1)^{n-3}x_1 \neq 0$. On the other hand, by
        Lemma~\ref{sylrankineq}, $\rank \varphi_0 + \rank \varphi_1 -
        (n-2) 
        \leq \rank (\varphi_0 \circ \varphi_1) = 0$, so $\rank
        \varphi_1 \leq n-3$. Consequently $\rank \varphi_1 = n-3$.   
	
	To compute $\rank \varphi_2$, we consider the minor of size $(n-3)^2$ 
        corresponding to rows $\{s+t(n-2) \, | \, 2 \leq s \leq n-2, 0 \leq t \leq n-4\}$ and columns $1,
        (n-2)+1, 2(n-2)+1, \ldots, ((n-3)^2-1)(n-2)+1$ which equals
        $x_1^{(n-3)^2} \neq 0$, so that $\rank
        \varphi_2 \geq (n-3)^2$. Again by Lemma~\ref{sylrankineq} and
        because $\varphi_1 \circ \varphi_2 = 0$, we know that $\rank
        \varphi_2 \leq (n-2)(n-3) - \rank \varphi_1 = (n-2)(n-3) -
        (n-3) = (n-3)^2$. We conclude that $\rank \varphi_2 = (n-3)^2$.  
	
	The rank computations for the remaining maps $\varphi_i$ follow
        easily from the block structure:
$\rank \varphi_i = (m-2)(n-3)^{i-1} + (n-3)(n-3)^{i-2}+ (p-2)(n-3)^{i-1} = (n-3)^i$ for any $i \geq 2$.  
	
	In the case of $\Phi_d$, we consider the minor corresponding
        to rows $2, 3, \ldots, d$ and columns $1, 1 + (n-2), 1 +
        2(n-2), \ldots, 1 + (d-2)(n-2)$ which equals
        $(-1)^{d-1}x_1^{d-1}$, so that $\rank \Phi_d \geq d-1$. On the
        other hand, because $\varphi_0 \circ \varphi_1 = 0$, $\Phi_d
        \circ \bigoplus\limits_{d-1} \varphi_1 = 0$. Therefore $\rank
        \Phi_d \leq dp-1)(n-2) - (d-1) \rank \varphi_1 = d-1$, and in
        fact $\rank \Phi_d = d-1$.  
\end{proof}

The ranks of the differentials of $\calF_\bullet(J)$ can be computed directly from Proposition~\ref{phirank}. 

\begin{corollary}\label{rankdelj}
	The ranks of the $\del_{J,i}$ are:
	\begin{enumerate}
		\item $\rank \del_{J,1} = n-2$ 
		\item $\rank \del_{J,2} = (n-2)(n-3)$
		\item $\rank \del_{J,i} = (n-2)(n-3)^{i-1}$ for all $i
                  \geq 2$ $\qed$
	\end{enumerate}
\end{corollary}

In order to apply Lemma~\ref{exactlem}, we need more information
regarding the ideals of maximal nonvanishing minors of the matrices involved.

\begin{proposition}\label{phiminors}
	For $i \in [n]$, 
        we have $x_i^{p-1} \in I(\Phi_p)$ and  $x_i^{(n-3)^{j}} \in
        I(\varphi_j)$  for all $j \geq 0$.
\end{proposition} 
\begin{proof}
	Because we are considering the ideals generated by the minors, we can 	  ignore signs in our
        computations. It is clear that $x_i \in I(\varphi_0)$ for all
        $i \in [n]$. To see that any $x_i^{n-3} \in I(\varphi_1)$ for
        any $i \in [n]$, we can consider the minors corresponding to the rows $r_{i,1}$ and columns $c_{i,1}$ listed below. 
	\[ 
	\begin{tabular}{l|p{1.5in}|p{3in}}
	$x_i^{n-3} \in I(\varphi_1)$ & rows $r_{i,1}$ & columns $c_{i,1}$ \\
	\hline 
	\hline 
	$1 \leq i \leq m-1$ & $2, 3, \ldots, n-2$ & $i + j(n-2)$ for $0 \leq j \leq n-4$ \\
	\hline 
	$i = m$ & $1, 2, \ldots, m-2, \newline m, \ldots, n-2$ & $(m-1) + j(n-2)$ for $0 \leq j \leq m-3$ \newline $\ell + (m-2)(n-2)$ for $m \leq \ell \leq n-2$ \\ 
	\hline 
	$i = m+1$ & $1,2, \ldots, m-1, \newline m+1, \ldots, n-2$ & $\ell+(m-2)(n-2)$ for $1 \leq \ell \leq m-1$ \newline $m+j(n-2)$ for $m-1 \leq j \leq n-4$\\ 
	\hline 
	$m+2 \leq i \leq n$ & $1, 2, \ldots, n-3$ & $i-2 + j(n-2)$ for $0 \leq j \leq n-4$ \\
	\end{tabular}
	\]
	
	The proposed submatrices of $\varphi_1$ whose rows and columns
        listed above are strictly triangular, so the minors are easily
        computed.  
	
	We can make a similar table with recipes for the appropriate
        minors in $\varphi_2$, given below. 
	\[ 
	\begin{tabular}{l|p{1.5in}|p{3.5in}}
	$x_i^{(n-3)^2} \in I(\varphi_2)$ & rows $r_{i,2}$ & columns $c_{i,2}$\\
	\hline 
	\hline 
	$1 \leq i \leq m-1$ & $s + t(n-2)$ for $s \in r_{i,1}$ and $0 \leq t \leq n-4$ & $i + j(n-2)$ for $0 \leq j \leq (n-4)(n-2)$ \\
	\hline 
	$i = m$ & $s + t(n-2)$ for $s \in r_{i,1}$ and $0 \leq t \leq n-4$ \newline $\backslash (m-1)(n-2)$ \newline $\cup \, (m-2)(n-1)+1$ & $a+j(n-2)(n-3)$ for $a \in c_{i,1}$ and $0 \leq j \leq (n-4)$ but $j \neq m-2$ \newline $(m-1) + \ell(n-2) + (m-2)(n-2)(n-3)$ for $0 \leq \ell \leq n-4$\\ 
	\hline 
	$i = m+1$ & $s + t(n-2)$ for $s \in r_{i,1}$ and $0 \leq t \leq n-4$ \newline $\backslash (m-2)(n-2)+1$ \newline $\cup \, m+(m-2)(n-2)$ & $a+j(n-2)(n-3)$ for $a \in c_{i,1}$ and $0 \leq j \leq (n-4)$ but $j \neq m-2$ \newline $m + \ell(n-2) + (m-2)(n-2)(n-3)$ for $0 \leq \ell \leq n-4$ \\ 
	\hline 
	$m+2 \leq i \leq n$ & $s +t(n-2)$ for $s \in r_{i,1}$ and $0 \leq t \leq n-4$ & $(i-2) + j(n-2)$ for $0 \leq j \leq (n-4)(n-2)$
	\end{tabular} 
	\]
	
	The block structures of the successive $\varphi_j$'s combined with the previous two statements is enough to see that $x_i^{(n-3)^j} \in I(\varphi_j)$. 
	
	Finally, use the minors whose columns and rows are given below to obtain $x_i^{d-1} \in I(\Phi_d)$. 
	
	\[ 
	\begin{tabular}{l|p{1.5in}|p{3.5in}}
	$x_i^{d-1} \in I(\Phi_d)$ & rows & columns\\
	\hline 
	\hline 
	$1 \leq i \leq m-1$ & $2, \ldots, d$ & $i + j(n-2)$ for $0 \leq j \leq d-2$ \\
	\hline 
	$i=m$ & $1, \ldots, d-1$ & $m-1 +j(n-2)$ for $0 \leq j \leq d-2$\\ 
	\hline 
	$i=m+1$ & $2, \ldots, d$ & $m + j(n-2)$ for $0 \leq j \leq d-2$\\
	\hline 
	$m+2 \leq i \leq n$ & $1, \ldots, d-1$ & $i-2 + j(n-2)$ for $0 \leq j \leq d-2$
	\end{tabular} 
	\]
\end{proof}

We are now ready to give the main result in this subsection.

\begin{theorem} 
	The complexes $\calF_\bullet(J)$ , $\calF_\bullet(I_1)$ and $\calF_\bullet(I_2)$ are exact.
\end{theorem}

\begin{proof}
       We only provide details for $\calF_\bullet(J)$. We
       show that we have exactness after localizing at any prime
       ideal of $R$, from which exactness over $R$ follows. If
       $\frakq$ is any prime ideal in $R$, we denote by
       $\del_{J,i,\frakq}$ the localized map induced by
       $\del_{J,i}$. The (unique) graded maximal ideal of $R$ is
       $\frakm =\langle x_1, \ldots, x_n \rangle$.

	Corollary~\ref{rankdelj} provides the ranks of the maps
        $\del_{J,i}$ over $R$. 
	Because $R$ is a domain, $I(\del_{J,i})$ contains exclusively
        non-zero divisors for all $i$. This means that, when
        localizing, the rank of $\del_{J,i}$ does not change. Furthermore,
        localization at any prime ideal $\frakq \neq \frakm$ yields
        $I(\del_{J,i,\frakq}) = R_\frakq$, because each $I(\del_{J,i})$ contains some power
        of every $x_\ell$, by Proposition~\ref{phiminors}. 
	By Lemma~\ref{exactlem}, this proves that $\calF_{\bullet}(J)$
        is exact after localization at any prime ideal $\frakq \neq
        \frakm$. 

We conclude that, if $\calF_{\bullet}(J)$ has a nonzero homology
module, it is only supported at the graded maximal ideal $\frakm$, and
therefore has depth $0$. 
Our goal now is to derive a contradiction. 
	
	We localize at $\frakm$, and use $F_i$ to denote the free
        $R_\frakm$-modules appearing in the localization of $\calF_\bullet(J)$.
        Use $B_i \subseteq C_i \subseteq F_i$
        to denote the $i$-cycles and $i$-boundaries, and $H_i=C_i/B_i$. The ring $R$ is a
        semigroup ring corresponding to a saturated (normal)
        semigroup, and is therefore Cohen--Macaulay by Hochster's theorem.
        Since $\dim R=3$, it follows that $R_\frakm$ has depth $3$.
        Consequently all the free modules over $R_\frakm$ also have depth
        $3$, in particular the $F_i$.
        Any submodules of the free modules $F_i$ must have depth at least $3$, so
        we have $\depth(C_i) \geq 3$ and
        $\depth(B_i) \geq 3$. From the exact sequence  
	\[
	0 \rightarrow B_i \rightarrow C_i \rightarrow H_i \rightarrow 0
	\]
	it follows that $\depth(H_i) \geq \min \{ \depth(C_i),
        \depth(B_i) - 1 \}$ (see~\cite[Corollary~18.6.a]{ecawav}), so that
        $\depth(H_i) \geq 2$. This contradicts that $\depth(H_i) = 0$. 
	
	Therefore localizations of $\calF_\bullet(J)$ at all prime
        ideals (now including $\frakm$) are exact, and consequently
        $\calF_{\bullet}$ is exact. 
\end{proof}

\subsection{The Mapping Cone}
\label{subsec:mappingCone}

We recall that we have an exact sequence
\[ 
0 \rightarrow J \xrightarrow{ \begin{bmatrix} \phantom{-}1 \\-1 \end{bmatrix} } I_1 \oplus I_2 \xrightarrow{ \begin{bmatrix} 1 & 1 \end{bmatrix} } \frakm \rightarrow 0. 
\] 
The relevant result for us is that, if we have resolutions of $J$ and $I_1
\oplus I_2$, the inclusion $J \to I_1 \oplus I_2$ can be lifted to a
map of complexes between the corresponding resolutions, and
the associated mapping cone is a resolution of $\frakm$.
The definition of the mapping cone of a map of complexes is given below; we
refer to the appendix of~\cite{ecawav} for more information. 

\begin{definition}
	If $\alpha: F \rightarrow G$ is a map of complexes, and we write $f,g$ for the differentials of $F$ and $G$ respectively, then the \emph{mapping cone} $M(\alpha)$ of $\alpha$ is the complex such that $M(\alpha)_{i+1} = F_i \oplus G_{i+1}$ where the differential $\del_{i+1}$ is shown: 
	\[ 
	\xymatrixrowsep{.5pc}
	\xymatrix{
		F_{i} \ar[r]^{-f_i} \ar[rdd]^{\alpha_i} & F_{i-1} \\
		\bigoplus & \bigoplus \\
		G_{i+1} \ar[r]^{g_{i+1}} & G_{i}
	}
	\]
	that is, $\del_{i+1}(a,b) = (-f_{i}(a), g_{i+1}(b) + \alpha_{i}(a))$. 
\end{definition}

We now construct the map of complexes that lifts the inclusion $J \to I_1 \oplus I_2$.

\begin{proposition}
	The map of complexes $\alpha: \calF_{\bullet}(J) \rightarrow \calF_{\bullet}(I_1) \oplus \calF_{\bullet}(I_2)$ is given by 
	\[ 
	\begin{aligned}
	\alpha_0 &= \left[ \begin{array}{cc}
	\begin{matrix} 
	x_{m+1} & 0 & \ldots & 0 & \multicolumn{3}{c}{\multirow{4}{*}{$\mathbf{0}$ }} \\ 
	0 & x_{m+1} & \ldots & 0 \\ 
	0 & \ddots & \ddots & 0 \\ 
	0 & 0 & \ldots & x_{m+1} &
	x_{m+2} & x_{m+3} & \ldots & x_n 
	\end{matrix}\\  
	\hline
	\begin{matrix} 
	-x_1 & -x_2 & \ldots & -x_{m} & 0 & 0 & \cdots & 0 \\ 
	\multicolumn{4}{c}{\multirow{3}{*}{$\mathbf{0}$ }} & -x_m & 0 & \cdots & 0 \\ 
	& & & & 0 & -x_m & \cdots & 0 \\ 
	& & & & 0 & \ddots & \ddots & 0 \\ 
	& & & & 0 & 0 & \cdots & -x_m
	\end{matrix}
	\end{array} \right] \in R^{n \times (n-1)} \\
	\alpha_i &= 
	\left[ \begin{array}{cc}
	x_{m+1}*\bbid_{(m-1)(n-2)(n-3)^{i-1}} & 0 \\ 
	0 & -x_m*\bbid_{(p-1)(n-2)(n-3)^{i-1}} 
	\end{array}
	\right] \in R^{(n-2)^2(n-3)^{i-1} \times (n-2)^2(n-3)^{i-1}} \textrm{ for } i \geq 1
	\end{aligned}
	\]
\end{proposition}

\begin{proof}	
	We first check that $\begin{bmatrix*}[r] 1 \\
          -1 \end{bmatrix*} \del_{J,0} = (\del_{I_1,0} \oplus
        \del_{I_2,0}) \alpha_0.$ We compute both sides explicitly:  
	\[
	\begin{aligned}
	\begin{bmatrix*}[r] 1 \\ -1 \end{bmatrix*} \del_{J,0}
        &=  \begin{bmatrix*}[r] 1 \\
          -1 \end{bmatrix*} \begin{bmatrix*} 
	x_1x_{m+1} & x_2x_{m+1} & \ldots& x_mx_{m+1} & x_mx_{m+2} &
        x_mx_{m+3} & \ldots & x_mx_n 
	\end{bmatrix*} \\ 
	&= \begin{bmatrix*}
	x_1x_{m+1} & x_2x_{m+1} & \ldots& x_mx_{m+1} & x_mx_{m+2} & x_mx_{m+3} & \ldots & x_mx_n \\
	-x_1x_{m+1} & -x_2x_{m+1} & \ldots& -x_mx_{m+1} & -x_mx_{m+2} & -x_mx_{m+3} & \ldots & -x_mx_n
	\end{bmatrix*} 
	\end{aligned}
	\]
	and 
	\[ 
	\begin{array}{ll}
	& (\del_{I_1,0} \oplus \del_{I_2,0}) \alpha_0 = \\ 
	&= \begin{bmatrix*}[r]
	x_1 & x_2 & \ldots & x_m & 0 & 0 & \ldots & 0 \\ 
	0 & 0 & \ldots & 0 & x_{m+1} & x_{m+2} & \ldots & x_n
	\end{bmatrix*} \left[ \begin{array}{cc}
	\begin{matrix} 
	x_{m+1} & 0 & \ldots & 0 & \multicolumn{3}{c}{\multirow{4}{*}{$\mathbf{0}$ }} \\ 
	0 & x_{m+1} & \ldots & 0 \\ 
	0 & \ddots & \ddots & 0 \\ 
	0 & 0 & \ldots & x_{m+1} &
	x_{m+2} & x_{m+3} & \ldots & x_n 
	\end{matrix}\\  
	\hline
	\begin{matrix} 
	-x_1 & -x_2 & \ldots & -x_{m} & 0 & 0 & \cdots & 0 \\ 
	\multicolumn{4}{c}{\multirow{3}{*}{$\mathbf{0}$ }} & -x_m & 0 & \cdots & 0 \\ 
	& & & & 0 & -x_m & \cdots & 0 \\ 
	& & & & 0 & \ddots & \ddots & 0 \\ 
	& & & & 0 & 0 & \cdots & -x_m
	\end{matrix}
	\end{array} \right] \\ 
	&= \begin{bmatrix*}
	x_1x_{m+1} & x_2x_{m+1} & \ldots& x_mx_{m+1} & x_mx_{m+2} & x_mx_{m+3} & \ldots & x_mx_n \\
	-x_1x_{m+1} & -x_2x_{m+1} & \ldots& -x_mx_{m+1} & -x_mx_{m+2} & -x_mx_{m+3} & \ldots & -x_mx_n
	\end{bmatrix*}
	\end{array}.
	\]
	
	Now we can check if $\alpha_0 \del_{J,1} = (\del_{I_1,1} \oplus \del_{I_2,1}) \alpha_1$. Without further ado: 
	\[ 
	\begin{array}{ll} 
	\alpha_0 \del_{J,1} &= \\ 
	&= \left[ \begin{array}{cc}
	\begin{matrix} 
	x_{m+1} & 0 & \ldots & 0 & \multicolumn{3}{c}{\multirow{4}{*}{$\mathbf{0}$ }} \\ 
	0 & x_{m+1} & \ldots & 0 \\ 
	0 & \ddots & \ddots & 0 \\ 
	0 & 0 & \ldots & x_{m+1} &
	x_{m+2} & x_{m+3} & \ldots & x_n 
	\end{matrix}\\  
	\hline
	\begin{matrix} 
	-x_1 & -x_2 & \ldots & -x_{m} & 0 & 0 & \cdots & 0 \\ 
	\multicolumn{4}{c}{\multirow{3}{*}{$\mathbf{0}$ }} & -x_m & 0 & \cdots & 0 \\ 
	& & & & 0 & -x_m & \cdots & 0 \\ 
	& & & & 0 & \ddots & \ddots & 0 \\ 
	& & & & 0 & 0 & \cdots & -x_m
	\end{matrix}
	\end{array} \right] \Phi_{n-1} \\ 
	&= x_{m+1}\Phi_{m} \oplus (-x_{m})\Phi_{p}
	\end{array}
	\] 
	where $x_{m+1}\varphi_0$ appears $m-1$ times and $-x_m \varphi_0$ appears $p-1$ times. However, because of the diagonal structure of $\alpha_1$, this is clearly the same as $(\Phi_m \oplus \Phi_{p}) \alpha_1 = (\del_{I_1,1} \oplus \del_{I_2,1}) \alpha_1$. 
	
	For remaining $i \geq 2$, $\del_{J,i} = \del_{I_1,i} \oplus \del_{I_2,i}$, and all the $\alpha_i$ are diagonal matrices, so the products are easily verified to be equal.
\end{proof}

\begin{proof}[Proof of Theorem~\ref{thm:resolution}]
Since $\calF_\bullet(J)$ is a resolution of $J$ and
$\calF_\bullet(I_1)\oplus\calF_\bullet(I_2)$ resolves $I_1\oplus I_2$,
the mapping cone of $\alpha: \calF_\bullet(J)\to
\calF_\bullet(I_1)\oplus\calF_\bullet(I_2)$ is a resolution of
$\frakm$. Augmenting the resolution to be a resolution of $\bbk =
R/\frakm$ results in a shift of one step, and so we finally have the
resolution $\calF_\bullet$. Comparing the rank of the free 
modules in each step to the Betti numbers computed in
Theorem~\ref{bettimain}, we conclude that $\calF_\bullet$ is not only
exact, but minimal.  
\end{proof}

\bibliographystyle{plain}
\bibliography{biblio} 

\end{document}